\documentclass{amsart}

\usepackage[margin=1.2in]{geometry}
\usepackage{amssymb,enumerate,amsmath,amsxtra,geometry}

\newtheorem*{Lem}{Lemma}

\newtheorem*{Def}{Definition}
\newtheorem*{Prop}{Proposition}

\newtheorem*{Cor}{Corollary}

\newtheorem*{ThmA}{Theorem A}
\newtheorem*{ThmB}{Theorem B}

\numberwithin{equation}{subsection}

\theoremstyle{remark}
\newtheorem*{Rmk}{Remark}

\newcommand{\uv}{{\rm U}(V)}
\newcommand{\guv}{{\rm GU}(V)}

\newcommand{\ep}{\epsilon}
\newcommand{\A}{\mathcal{A}}
\newcommand{\p}{\mathfrak{p}}

\newcommand{\et}{E[T]}
\newcommand{\etl}{E[T,T^{-1}]}

\newcommand{\vform}{\langle\mspace{7mu},\mspace{7mu}\rangle}

\newcommand{\tp}{{}^{\top}}

\newcommand{\pic}{\pi^{\vee}}
\newcommand{\pia}{\pi^{\alpha}}
\newcommand{\pii}{\pi^{\iota}}
\newcommand{\pit}{\pi^{\theta}}

\newcommand{\ua}{{}^{\alpha}}
\newcommand{\ui}{{}^{\iota}}
\newcommand{\ut}{{}^{\theta}}

\newcommand{\cx}{\mathbb{C}}
\newcommand{\tr}{{\rm tr}}
\newcommand{\fc}{f^{\vee}}

\newcommand{\fo}{\mathfrak{o}_F}

\newcommand{\fp}{\mathfrak{p}_F}

\newcommand{\nrd}{{\rm Nrd}}

\newcommand{\dx}{D^\times}

\newcommand{\dop}{D^{\texttt{o}}}

\newcommand{\fx}{F^\times}

\newcommand{\ompi}{\omega_\pi}

\newcommand{\od}{\mathfrak{O}}

\newcommand{\pd}{\mathfrak{q}}

\begin{document}

\thanks{}

\title{A factorization result for classical and similitude groups}

\author[A. Roche]{Alan Roche}
\address{Dept. of Mathematics, University of Oklahoma, Norman OK 73019-3103.}
\email{aroche@math.ou.edu}
\author[C.~Ryan Vinroot]{C.~Ryan Vinroot}
\address{Dept. of Mathematics, College of William and Mary, P.O. 8795,  Williamsburg, VA 23187-8795.}
\email{vinroot@math.wm.edu}


\begin{abstract}
For most classical and similitude groups, we show that each element can be written as a product of two transformations that 
a) preserve or almost preserve the underlying form and b) whose 
squares are certain scalar maps. This generalizes work of Wonenburger and Vinroot.   
As an application, we re-prove and slightly extend a well-known result of M{\oe}glin, Vign\'{e}ras and Waldspurger  on the existence of automorphisms of  $p$-adic classical groups that take each irreducible smooth representation to its dual. 
\end{abstract}

\maketitle 

\section*{Introduction} 
For many classical groups $G$, we show that each element is a product of two involutions. 
The involutions belong to a group  $\widetilde{G}$ containing $G$ such that  $[\widetilde{G}:G] \leq 2$. 
We also prove a similar factorization for elements of the corresponding similitude groups. 
Our interest in such factorizations stems from an application to the representation theory of reductive groups over 
non-archimedean local fields. We are interested in involutary automorphisms of such groups that take each irreducible smooth representation to its dual. Echoing \cite{Ad}, we call these {\it dualizing} involutions. They do not always exist in our setting (we give an example in \S\ref{non-existence}). They do exist, however, for many classical $p$-adic groups by a result of  M{\oe}glin, Vign\'{e}ras and Waldspurger 
 (\cite{MoViWa87} Chap.~IV \S~II). We re-prove this result and slightly extend its scope as explained below. 

To make more precise statements, we need to define the classical and similitude groups we consider. 
Let $E/F$ be a field extension with $E=F$ or $[E:F]=2$. We assume in the quadratic case that $E/F$ is a Galois extension. In all cases
we write $\tau$ for the generator of  ${\rm Gal} (E/F)$, so that $\tau$ has order two when $[E:F] = 2$ and  $\tau = 1$ when $E=F$. 
Let $V$ be a finite dimensional vector space over $E$ with a non-degenerate $\ep$-hermitian form $\vform$ ($\ep = \pm 1$)
which we take to be linear in the first variable. Thus 
\[
       \langle \alpha u +  \beta v, w \rangle =   \alpha \langle u, w  \rangle + \beta \langle  v,w \rangle \,\,\,\text{and} \,\,\,
          \langle v, w \rangle = \ep \, \tau ( \langle w, v \rangle )          
       \]
for all $\alpha, \beta \in E$ and $ u, v, w \in V$. It follows that $\vform$ is $\tau$-linear in the second variable: 
\[
       \langle u,  \alpha v +  \beta w \rangle =   \tau(\alpha)  \langle u, v  \rangle + \tau(\beta)  \langle  u , w \rangle.
       \] 
 In the case $\text{char} \,F = 2$ and $E=F$ we assume also that 
$\langle v, v \rangle = 0$ for all $v \in V$, that is, $\vform$ is symplectic.

We write $\uv$ for the isometry group (or unitary group) of $\vform$ 
and 
$\guv$ for the corresponding similitude group. That is, 
\begin{align*} 
   \uv  &= \{ g \in {\rm Aut}_E (V) :  \langle g v, g v' \rangle = \langle v, v' \rangle, \,\,\, \forall \,\, v, v' \in V \}, \\
      \guv  &= \{ g \in {\rm Aut}_E (V) :  \langle g v, g v' \rangle =  \beta \langle v, v' \rangle, \, \text{for some scalar $\beta$}, 
 \,  \forall \,\, v, v' \in V \}.
\end{align*} 
For $g \in \guv$, applying $\tau$ to both sides of $ \langle g v, g v' \rangle =  \beta \langle v, v' \rangle$ gives $\tau(\beta) = \beta$,  so that 
$\beta \in F^\times$.
We write $\mu(g) = \beta$. It is called the {\it multiplier} of $g$ and the resulting 
homomorphism $\mu:\guv \to F^\times$ is the {\it multiplier} map.  
\begin{Def}
{\rm Let $h \in {\rm Aut}_F(V)$. 
We say that $h$ is {\it anti-unitary} if  
\[
           \langle h v, h v' \rangle    =      \langle   v', v \rangle, \quad \quad \forall \,\, v, v' \in V.
           \]
}
 \end{Def}        
\noindent   
When  $E=F$ and $\text{char} \,F \neq 2$,  the form $\vform$ is orthogonal ($\ep = 1$) or symplectic ($\ep = -1$). 
In the orthogonal case, an anti-unitary map is simply an orthogonal transformation. 
In the symplectic case, an anti-unitary map is a skew-symplectic transformation 
($\langle h v , h v' \rangle = - \langle v, v' \rangle$).  

We also need the corresponding notion for similitude groups. 
\begin{Def}
{\rm Let $h \in {\rm Aut}_F(V)$. 
We say also that $h$ is an {\it anti-unitary similitude} if, for some scalar $\beta$,            
\[
\langle h v, h v' \rangle    =  \beta     \langle   v', v \rangle, \quad \quad \forall \,\, v, v' \in V. 
 \]    
}
\end{Def}
\noindent 
Thus an anti-unitary map is an anti-unitary similitude for which $\beta  = 1$. 

We can now state our factorization result. 
\begin{ThmA}
Let $g \in \guv$ with $\mu(g) = \beta$. Then there is an anti-unitary involution $h_1$ 
and an anti-unitary similitude $h_2$ with $h_2^2 = \beta$ such that $g = h_1 h_2$.
In particular, for any $g \in \uv$, there exist anti-unitary
elements $h_i$ with $h_i^2 = 1$ (for $i=1,2$) such that $g = h_1 h_2$.  
\end{ThmA}

For example, Theorem~A says that any orthogonal transformation is a product of two orthogonal involutions and that 
any symplectic transformation is a product of two skew-symplectic involutions.  This was 
originally proved by Wonenburger \cite{Wo}  (under the assumption $\text{char}\,F \neq 2$). 
While  we ultimately obtain a new proof of her results, we borrow heavily from her approach. 
In particular, the arguments in \S \ref{Case II} below are in essence those of \cite{Wo} but rephrased in the language of modules. 
For $E=F$ and $\text{char}\,F \neq 2$, Theorem~A is due in the case of similitude groups to Vinroot \cite{ryan-gsp, ryan-go} 
(by an adaptation of Wonenburger's arguments).

Our framework does not accommodate orthogonal groups in even characteristic
(defined as the stabilizers of suitably non-degenerate quadratic forms) or the corresponding 
similitude groups. If $F$ is perfect, then it follows readily from work of Gow \cite{Gow} or Ellers and Nolte \cite{EN} that 
Theorem~A continues to hold in this setting.

Suppose now that  $F$ is a non-archimedean local field and that $G$ is the 
group of $F$-points of a reductive $F$-group. Let $\pi$ be an 
irreducible smooth representation of $G$. For any continuous automorphism $\alpha$ of $G$, 
we write $\pia$ for the (smooth) representation of $G$ given by $\pia (g) = \pi (\ua g)$ for $g \in G$.  
We write $\pic$ for the smooth dual or contragredient of $\pi$. 

\begin{Def}
{\rm Let $\iota$ be a  continuous automorphism of $G$ of order at most two. We say that 
$\iota$ a {\it dualizing involution}
of $G$ if $\pii \cong \pic$ for all irreducible smooth representations $\pi$ of $G$.}
\end{Def}

We fix an anti-unitary involution $h \in {\rm Aut}_F(V)$ and set $\ui g = \mu(g)^{-1} hgh^{-1}$ for $g \in \guv$. Then $\iota$ defines a continuous
automorphism of $\guv$ of order two. Further $\iota |_{ \uv}$ gives the automorphism  $g \mapsto hgh^{-1}$ of $\uv$ 
which for simplicity we again denote by $\iota$.  Our application of Theorem~A hinges on the following immediate consequence. 
\begin{Cor}
For any $g \in \guv$, the elements $\ui g$ and $g^{-1}$ are conjugate by an element of $\uv$.
\end{Cor}
\begin{proof}
Let $a \in \guv$ with $\mu(g) = \beta$.  
By Theorem~A, we have $g = h_1 h_2$ for an anti-unitary involution $h_1$ and an anti-unitary similitude $h_2$ with $h_2^2 = \beta$. 
Thus $h_2^{-1} = \beta^{-1} h_2$ and $g^{-1} = \beta^{-1} h_2 h_1$. Hence 
\begin{align*} 
     (h_1 h )\, \ui g \, (h_1 h)^{-1}  &=  h_1 h \, ( \beta^{-1} h (h_1 h_2) \, h^{-1} ) \, h h_1  \\  
                                            &=  \beta^{-1} h_2 h_1 \\ 
                                               &= g^{-1}.
                                                \end{align*}
That is, $\ui g$ and  $g^{-1}$ are conjugate by $h_1 h \in \uv$. 
\end{proof}
For the classical groups $\uv$, this is part of \cite{MoViWa87}  Chap.\,4  Prop.\,I.2 and the early part of our proof of Theorem~A mirrors the 
corresponding part of the proof in {\it loc.~cit.} Our overall proof, on the other hand, can easily be adapted to give an alternative route to the full statement in {\it loc.~cit.}

Our main result is the following. 
\begin{ThmB}  
The maps $\iota:\uv \to \uv$ and $\iota:\guv \to \guv$ are dualizing involutions. 
\end{ThmB}
In the case of the classical groups $\uv$, this is essentially \cite{MoViWa87} Chap.~IV  Th\'{e}or\`{e}me~II.1. 
Given Harish-Chandra's theory of characters \cite{H-C, Ad-Kor} as recalled in \S\ref{main result}, Theorem~B is an immediate consequence 
of the Corollary. 

The argument in  \cite{MoViWa87} does not rely on existence of characters. Instead it adapts a geometric method used by  
Gelfand and Kazhdan to show that transpose-inverse is a dualizing involution of ${\rm GL}_n(F)$ \cite{GK}.  As above, 
Gelfand and Kazhdan's result follows immediately from the existence of characters. Indeed, by elementary linear algebra, 
any square matrix is conjugate to its transpose. It follows that if $\ut g = {}^{\top} g^{-1}$ for $g \in  {\rm GL}_n(F)$ then, 
for any irreducible smooth representation $\pi$, the characters of $\pit$ and $\pic$ are equal, whence $\pit \cong \pic$. 
Tupan found a clever and completely elementary proof of Gelfand and Kazhdan's result \cite{Tupan}.
We will report in a sequel  on a similarly elementary proof of Theorem~B that builds on 
Tupan's approach \cite{RV}.

Finally, let $G$ be the isometry group of a non-degenerate hermitian or anti-hermitian form over a $p$-adic quaternion algebra.  
By  \cite{LinSunTan}, there is no automorphism $\theta$ of $G$ such that ${}^{\theta} g$ is 
conjugate to $g^{-1}$ for all $g \in G$. Thus the Corollary above is false in this setting which means surely that 
Theorem~B does not extend to classical groups over $p$-adic quaternion algebras.  In this spirit, let $D$ be a central (finite-dimensional) 
division algebra over $F$. 
By a straightforward argument involving only central characters, due to the first-named author and Steven Spallone, 
the group $\mathrm{GL}_n(D)$ can admit an 
automorphism that takes each irreducible smooth representation to its dual only in the known cases $D=F$ and when $D$ is a quaternion 
algebra over $F$ \cite{MS, Ragh}. In particular, in contrast to the case of connected reductive groups over the reals 
\cite{Ad}, dualizing involutions in our sense do not always exist. 

{\it Organization.}  The proof of Theorem~A takes up  \S\S1 through 5. We record some special cases and applications in \S6.  
In \S7 we briefly recall some character theory and prove Theorem~B. Finally in \S8 we show that the unit groups of 
(finite-dimensional) central simple algebras over $F$ do not admit dualizing involutions except in the two cases noted above.

\section{Proof of Theorem~A: Initial Setup and First Reduction}
\noindent
{\it Notation.} Let $R$ be a ring with identity. We write $R^\times$ for the group of units of $R$. 
For any $R$-module $M$ (which for us is always a unital left $R$-module), we write $\text{ann}_R \,M$ for the annihilator of $M$. That is, 
\[
    \text{ann}_R \, M = \{ r \in R :  r m = 0,  \,\,\, \forall \,\, m \in M \}.
    \]
For $m \in M$, we also write $\text{ann}_R \, m = \{ r \in R : r  m = 0 \}$. 
Thus $\text{ann}_R \,M = \bigcap_{m \in M} \text{ann}_R \, m$. 
Note that $\text{ann}_R \, m$ is the kernel of the surjective $R$-module homomorphism
$r \mapsto r m: R \to Rm$, so that $ R/ \text{ann}_R \, m \cong   Rm$ as $R$-modules.

 \subsection{} 
Let $g \in \guv$ with $\mu(g) = \beta$. The space $V$  is a module over the polynomial ring 
$\et$ via $f(T) v = f(g) v$.  Let $p = p(T)$ denote the minimal polynomial of $g$. We have 
\[
  p = p_1^{e_1} \cdots p_n^{e_n}
\]
for distinct monic irreducible elements $p_1, \ldots, p_n \in \et$ and positive integers $e_1, \ldots, e_n$. 

We set $\A = \et / (p)$.  The ideal $(p)$ is simply the annihilator of $V$ as an $\et$-module. In particular,  
$V$ carries an induced  $\A$-module structure.  
The Chinese Remainder Theorem gives a canonical isomorphism of $E$-algebras 
\[
   \et/ (p)   \cong \et/(p_1^{e_1}) \oplus \cdots \oplus \et / (p_n^{e_n}).
\]
Thus 
\[
   \A =  \A_1 \oplus \cdots \oplus \A_n, 
   \]
for ideals $\A_i$ in $\A$ with $\A_i \cong \et/ (p_i^{e_i})$ ($i=1, \ldots, n$). Setting $V_i = \A_i V$ ($i=1, \ldots, n$), 
we have 
\begin{equation}  \label{primary-decomp}
   V =    V_1 \oplus \cdots \oplus V_n.
\end{equation} 
Each $V_i$ is an $E[T]$-submodule and as such has annihilator $(p_i^{e_i})$.  More concretely, each $V_i$ is $g$-stable
and the minimal polynomial of $g$ on $V_i$ is $p_i^{e_i}$.

\subsection{}
As $g$ is invertible, the $\et$-module structure on $V$ extends to a module structure over the ring of Laurent polynomials
$\etl$.  It follows that each $V_i$ in (\ref{primary-decomp}) is an $\etl$-submodule.  
We have 
\[
\text{ann}_{\etl} \,V = p\,\etl \,\,\,  \text{and} \,\,\,  \text{ann}_{\etl} \,V_i = p_i^{e_i}\,\etl  \hspace{12pt} (i=1, \ldots, n).
\] 
The inclusion $\et \subset \etl$ induces canonical $E$-algebra isomorphisms 
\[
\et / (p) \cong  \etl / p \, \etl \,\,\,   \text{and} \,\,\, \et/ (p_i^{e_i}) \cong \etl / p_i^{e_i} \, \etl   \hspace{12pt} (i=1, \ldots, n).
\]
We use these to identify $\A$ with $\etl / p \, \etl$ and each  $\A_i$ with $\etl / p_i^{e_i} \, \etl$.

The $F$-automorphism $\tau$ of $E$ extends to an involution  
\[
    \sum_i   a_i T^i  \longmapsto \sum_i   \tau (a_i)  \beta^i T^{-i}
    \]
on $\etl$ which we continue to denote by $\tau$.  This satisfies the adjoint relation  
\begin{equation} \label{adjoint-tau} 
   \langle  v, f w \rangle =    \langle \tau(f) v, w \rangle, \quad \quad \forall \,\, v, w \in V, \,\,\,\forall \,\, f \in \etl.
\end{equation} 
It follows that  $\tau ( p \,\etl ) = p \,\etl$. Hence there is a $u \in \etl^\times$  
such that $ \tau (p) = u p$ and thus $\tau$ induces an involution on $\A$. 

Further,  for $i= 1, \ldots, n$,    
\[ 
({\rm I} )  \,\,\,  \tau(p_i) = u_i p_{i'} \,\,\,\text{for} \,\,\, i' \neq \, i  \quad \text{or} \quad  ({\rm II}) \,\, \, \tau(p_i) = u_i p_i 
\]
with each $u_i \in \etl^\times$. In case $({\rm I})$ $ \tau$ induces an isomorphism $\A_i \cong \A_{i'}$ while in case $({\rm II})$ it  induces an involution on $\A_i$. 

By (\ref{adjoint-tau}),  
\begin{equation}   \label{perp}
V_k \perp V_l  \,\,\,  \text{unless} \,\,\, \tau (p_k) = u p_l \,\,\,\text{for some}   \,\,\,u \in \etl^\times. 
\end{equation} 
It follows that 
\[
   V = W_1 \oplus \cdots \oplus W_m
   \]
where for a given $W_j$ we have $W_j = V_i \oplus V_{i'}$ for some $i$ and $i'$ as in $({\rm I})$ above or $W_j = V_i$ 
with $i$ as in $({\rm II})$. In particular,  each $W_j$ is an $\etl$-submodule and 
the restriction of $\vform$ to each $W_j$ is non-degenerate. Thus $g \in \guv$ decomposes as 
$g = g_1 \oplus \cdots \oplus g_m$ with $g_j \in GU(W_j)$ for $j= 1,\ldots, m$.  It suffices to prove the result for each 
$g_j$. This means we are reduced to two basic cases. 

\smallskip
{\bf Case I.}  The minimal polynomial of $g$ is $ p_1^e p_2^e$ for some positive integer $e$ and monic irreducible polynomials $p_1, p_2 \in \et$ such that  $\tau(p_1) = u p_2$ for some $u \in \etl^\times$.  We have $\A = \A_1 \oplus \A_2$ with 
\[
\A_i = \et / (p_i^e) =  \etl / \, p_i^e \, \etl  \quad \quad (i=1,2). 
\]
The space $V$ decomposes as $V = V_1 \oplus V_2$ where $V_i = \A_i V$ ($i=1,2$). 
Moreover, by (\ref{perp}), each  $V_i$ is a totally 
isotropic subspace of $V$.

\smallskip 
{\bf Case II.} The minimal polynomial of $g$ is $p^e$ for some positive integer $e$ and some monic irreducible element 
$p \in \et$ such that $\tau(p) = u p$ for some $u \in \etl^\times$. In this case,
\[
\A   = \et / (p^e) = \etl / \, p^e\,\etl.
\]

\section{Proof of Theorem~A: Case I}
\subsection{}  \label{twist} 
As $V=V_1 \oplus V_2$ is non-degenerate and each $V_i$ is totally isotropic, it follows that $\vform$ induces an isomorphism
between $V_1$ and the conjugate dual of $V_2$. That is, if  we write $V_2^\tau$ for the vector space structure on $V_2$ obtained by 
twisting by $\tau$ so that $V_2^\tau = V_2$ as  abelian groups and scalar multiplication on $V_2^\tau$
is given by $a.v = \tau(a) v$ (for $a \in E$ and $v \in V_2$), then  
\[
v \mapsto \langle v, - \rangle: V_1 \longrightarrow \text{Hom}_E(V_2^\tau, E)
 \]
 is an isomorphism of $E$-vector spaces.

Let $e_1, \ldots, e_n$ be any basis of $V_1$. By the preceding paragraph, $V_2$ (or $V_2^\tau$) admits a 
dual basis $f_1, \ldots, f_n$ such that
\[
   \langle e_i, f_j \rangle = \begin{cases} 
                                                         1, \,\,\,\text{if} \,\, i = j, \\
                                                         0, \,\,\,\text{if} \,\, i \neq j.
                                                         \end{cases}
                                                         \]
Thus with respect to the basis $e_1, \ldots, e_n, f_1, \ldots, f_n$, the matrix of $\vform$ is given in block form by 
\[
     J = \begin{bmatrix} 
                   0    & \ep  I_n \\
                   I_n     &  0
                   \end{bmatrix}.
                   \]
For any matrix $a = [a_{ij}]$ with entries in $E$, we set $\tau(a) = [\tau(a_{ij})]$ and write $\tp a$ for the transpose of $a$.                    
Below we often view $E$-linear maps on $V$ as (block) matrices with respect to the 
basis $e_1, \ldots, e_n, f_1, \ldots, f_n$.  

Consider the $F$-linear map $c:V \to V$ given by 
 \[
    \sum_{i=1}^n a_i e_i + \sum_{j=1}^n b_j f_j \overset{c}{\longmapsto}      
                 \sum_{i=1}^n \ep \, \tau(a_i)  e_i + \sum_{j=1}^n \tau(b_j)  f_j.
                 \]      
Setting $a = \begin{bmatrix} 
                                  a_1 \\
                                    \vdots \\
                                  a_n
                                  \end{bmatrix}$ 
                                  and 
                 $b = \begin{bmatrix} 
                                  b_1 \\
                                    \vdots \\
                                  b_n
                                  \end{bmatrix}$, 
we can write $c$  in matrix form as 
$ \begin{bmatrix}
             a \\
             b
             \end{bmatrix} 
             \overset{c}{\longmapsto}  
                                      \begin{bmatrix}
             \ep\, \tau(a)  \\
                 \tau(b)
             \end{bmatrix} $.
The map $c$ is anti-unitary (that is, $\langle c(v) , c(v') \rangle = \langle v',v \rangle$, for all $v, v' \in V$)  and $c^2 = 1$. 

Any anti-unitary $h_1 \in {\rm Aut}_F(V)$ can be written as $h_1 = s_1  c$ with $s_1  \in U(V)$. 
Now  $h_1 = s_1c $ is an involution if and only if $s_1 \,{}^{c}s_1 = 1$ where
${}^{c}s_1 = c s_1 c^{-1}$.  
Similarly, an anti-unitary similitude $h_2$ can be written as $h_2 = c s_2$ with $s_2 \in \guv$. Again
$h_2^2 = \beta$ if and only if  $s_2 \,{}^{c}s_2 = \beta$ with ${}^{c} s_2  =  c s_2 c^{-1}$. 
In this notation, we have $h_1 h_2 = s_1 s_2$ (as $c^2 =1$). 
It follows that Theorem~A in Case~I is equivalent to the following: 
\begin{enumerate}[$(\ast)$]
\item
if $g \in \guv$ with $\mu(g) = \beta$ then $g =   s_1 s_2$ for elements $s_1 \in U(V)$ and $s_2 \in \guv$ such that $s_1 \,{}^{c}s_1 = 1$
and $s_2 \,{}^{c}s_2 = \beta$. 
\end{enumerate}

\subsection{}   \label{frob}
We now prove $(\ast)$. 
Since $g$ preserves $V_1$ and $V_2$, we have 
$g = \begin{bmatrix} 
            a    &     0 \\
            0     &     b
            \end{bmatrix}$. 
The condition $g \in \guv$ says $\tp g J \tau(g) = \beta J$ with $\beta = \mu(g)$. 
A short matrix calculation shows that this means $b = \beta \, \tp \tau(a)^{-1}$, so 
that 
 \[
 g = \begin{bmatrix} 
            a    &     0 \\
            0     &    \beta \, \tp \tau(a)^{-1}
            \end{bmatrix}.
            \]
We set
\[
s_1 = \begin{bmatrix} 
              0   &   d_1 \\
             \ep\, \tp \tau(d_1)^{\,-1}  & 0            
            \end{bmatrix}, \,\,\,
            s_2 = \begin{bmatrix} 
              0     &    \ep\,\beta \, \tp \tau(d_2)^{-1} \\
              d_2     & 0            
            \end{bmatrix},  
            \]
for elements $d_1, d_2 \in {\rm GL}_n(E)$.  It is routine to check that  
$\tp s_1 J \tau(s_1) = J$ and $\tp s_2 J \tau(s_2) = \beta \, J$. Thus $s_1 \in \uv$ and $s_2 \in \guv$.  

To calculate ${}^c s_1$, note that for all column vectors 
$ \begin{bmatrix} 
            x \\
            y 
            \end{bmatrix}$ as above,  we have 
\begin{align*} 
     \begin{bmatrix} 
            x \\
            y 
            \end{bmatrix}   \overset{c}{\longmapsto} 
      \begin{bmatrix} 
          \ep \,\tau(x) \\
              \tau(y) 
              \end{bmatrix}   \overset{s_1}{\longmapsto}  
         \begin{bmatrix} 
              0   &   d_1 \\
             \ep\, \tp \tau(d_1)^{\,-1}  & 0            
            \end{bmatrix}  \,            
    \begin{bmatrix} 
          \ep \,\tau(x) \\
              \tau(y) 
              \end{bmatrix}     &=
         \begin{bmatrix} 
            d_1 \, \tau(y) \\
          \tp \tau(d_1)^{\,-1} \, \tau(x)  
              \end{bmatrix}   \\
          & \overset{c}{\longmapsto}
          \begin{bmatrix} 
            \ep\,\tau(d_1) \, y \\
          \tp d_1^{\,-1} \, x 
              \end{bmatrix}   
          =  \begin{bmatrix}
                   0   &  \ep\,\tau(d_1) \\
                   \tp  d_1^{\,-1}   & 0
                   \end{bmatrix} \,
                   \begin{bmatrix} 
                       x \\
                       y
                       \end{bmatrix}.            
 \end{align*}
 That is, 
 \[
      {}^c s_1  =    \begin{bmatrix}
                   0   &  \ep\,\tau(d_1) \\
                   \tp d_1^{\,-1}   & 0
                   \end{bmatrix}   =   \ep \,\tau(s_1).
                   \]
A similar computation gives 
\[
        {}^c s_2     =    \begin{bmatrix}
                   0   &  \beta \, \tp d_2^{\,-1} \\
                   \ep \, \tau (d_2)   & 0
                   \end{bmatrix}   =   \ep \,\tau(s_2).
                   \]                  
                             
By direct matrix calculations, the conditions $s_1 \, {}^{c}s_1 = 1$ and   $s_1 \, {}^{c}s_1 = \beta$ 
are equivalent to 
\[
    d_1 \,\tp  d_1^{\,-1} = I_n \,\,\,\text{and} \,\,\, d_2 \,\tp  d_2^{\,-1} = I_n, 
    \]
that is, $d_1$ and $d_2$ are symmetric.  
Since  $g = s_1 s_2$ is equivalent to $a = d_1 d_2$, we are reduced to the matrix statement:
\begin{enumerate}[$(\ast)'$]
\item
For any (invertible) $n \times n$ matrix $a$ (with entries in $E$), 
there exist (invertible) symmetric $n \times n$ matrices $d_1$ and $d_2$ 
(with entries in $E$) such that 
$
a    = d_1  d_2.
$
\end{enumerate}
Now it is well-known that any square matrix is conjugate to its transpose by a symmetric matrix (see, for example, \cite{Kap} page~76).  
Thus
\[
           d^{\,-1}  a  d   =  \tp  a
           \]
 with $d \in {\rm GL}_n(E)$ symmetric.           
This means $d^{\,-1} \, a  =  \tp a \, d^{\,-1}$, so $\tp (d^{\,-1} \, a) = \tp a \,d^{\,-1} = d^{\,-1} \,a$.           
Therefore  
\[
      a    = d  \cdot d^{\,-1} a
      \]
expresses $a$ as product of symmetric matrices (with entries in $E$). This completes the proof of Theorem~A in Case~I.

\section{Proof of Theorem~A: Case II and Second Reduction}  \label{Case II}
In this case, the minimal polynomial of $g$ is $p^e$  (for some positive integer $e$) 
where $p$ is irreducible and  $\tau(p) = up$ for some $u \in \etl^\times$.  
Let  $\A = \etl / p^e  \etl$. As $\text{ann}_{\etl} \,V = p^e\,\etl$, 
the space $V$ is naturally an $\A$-module and as such is faithful, that is, 
$\text{ann}_{\A} \,V = \{ 0 \}$.  Note that $\A$   
is a local ring with unique maximal ideal 
$\p = p \etl / p^e \etl$. 
More strongly, the ideals in $\A$ form a chain 
\[
\A \supsetneqq \p \supsetneqq \cdots \supsetneqq \p^{e-1} \supsetneqq \p^e = \{ 0 \}.
\]

\subsection{}   \label{non-deg criterion}
As $\text{ann}_{\A} \,V = \{ 0 \}$, there is some $v \in V$ such that 
$\text{ann}_{\A} \,v = \{ 0 \}$.   
Below we will need to consider the restriction of $\vform$ to the submodule $\A v$ 
generated by such an element and will make use of the following non-degeneracy criterion. 

\begin{Lem}
Let $v \in V$ with $\text{ann}_{\A} \, v = \{ 0 \}$ (equivalently,  $\text{ann}_{\etl} \, v = p^e \,\etl$).  The cyclic submodule
$\A v$ is non-degenerate 
if and only if $\langle \p^{e-1} v, v \rangle \neq \{ 0 \}$. 
\end{Lem}

\begin{proof} 
$(\Rightarrow)$ Suppose $\A v $ is non-degenerate. By hypothesis, $p^{e-1}v \neq 0$. Thus there is an $f  \in \etl$ such that 
$ \langle p^{e-1} v, f v \rangle \neq 0$, so that $\langle \tau(f) \, p^{e-1} v, v \rangle \neq 0$ and hence 
$\langle \p^{e-1} v, v \rangle \neq \{ 0 \}$. 

$(\Leftarrow)$ Suppose now that $\langle  \p^{e-1} v, v \rangle \neq \{ 0 \}$. 
We write $\text{rad} \,\A v$ for the radical of $\vform$ on restriction to $\A v$. It is immediate that 
$\text{rad} \,\A v$ is an $\A$-submodule. The map $a \mapsto av:\A \to \A v$ is an isomorphism of $\A$-modules. It follows that 
$\text{rad} \,\A v = \p^c v$ for some non-negative integer $c$ (as the only ideals in $\A$ are the powers of $\p$). 
Our assumption $\langle \p^{e-1} v, v \rangle \neq \{ 0 \}$ implies that $c > e-1$. Thus 
$\text{rad} \,\A v = \{ 0 \}$, that is, $\A v$ is non-degenerate.
\end{proof}

\subsection{}   \label{x and y}
Let $x \in V$ with $\text{ann}_{\A} \, x = \{ 0 \}$ (equivalently,  $\text{ann}_{\etl} \,x = p^e \, \etl$)  and set $X = \A x$.  
Now $p^{e-1} x \neq 0$, so there is a $y \in V$ with $\langle p^{e-1} x, y \rangle \neq 0$. 
Assume that the subspace $X$ is degenerate, so that $y \notin X$ (by Lemma~\ref{non-deg criterion}). 
Setting $Y = \A y$, we claim that if $X \cap Y \neq \{ 0 \}$ then $Y$ is non-degenerate.

To prove this, let $z \in X \cap Y$. We have 
\[
    z   = p^c g x = p^{c'} g' y
   \]      
for integers $c$ and $c'$ with $0 \leq c < e,\, 0 \leq c' < e$ and elements $g, g' \in \et$ that are prime to $p$. 
Thus 
\[
    \text{ann}_{\et} \, z = ( p^{e-c}) = (p^{e-c'})
    \]
and so $c = c'$.        

Now there are elements $a, b \in \et$ such that $a g + b p^e = 1$. Hence 
\begin{align*} 
     p^{e-c-1} az      &= p^{e-c-1} a ( p^c g x) \\
                               &=  p^{e-1} ag x \\
                               &=  p^{e-1} ( 1- bp^e) x \\
                               &= p^{e-1} x  \quad (\text{as}\,\,p^e x = 0). 
                               \end{align*} 
 In addition, 
 \begin{align*} 
      p^{e-c-1} az     &= p^{e-c-1} a ( p^c g' y) \\
                               &= p^{e-1} a g' y, 
                               \end{align*}                               
so that 
\[
        p^{e-1} x    =   p^{e-1} a g' y.
       \]
As   $\langle p^{e-1} x, y \rangle \neq 0$, it follows that $\langle p^{e-1} a g' y, y \rangle \neq 0$. Therefore 
$\langle \p^{e-1}y, y \rangle \neq \{ 0 \}$. Hence, by Lemma~\ref{non-deg criterion}, $Y = \A y$ is non-degenerate

\subsection{} \label{non-cyclic} 
We now show that  if $V$ does not admit a non-degenerate cyclic submodule (generated by 
an element $v$ such that $\text{ann}_{\A} \, v = \{ 0 \}$) then it must contain a non-degenerate non-cyclic submodule of 
a very special kind. 

\begin{Lem}
Suppose that for any $v \in V$ such that $\text{ann}_{\A}\,v = \{ 0 \}$ the submodule $\A v$ is degenerate. Then there exist
$x$ and $y$ in $V$ such that $\langle \p^{e-1} x, y \rangle \neq \{ 0 \}$. We have $\A x \cap \A y = \{ 0 \}$ and the submodule
$\A x \oplus \A y$ is non-degenerate. 
\end{Lem}

\begin{proof}
As in \S\ref{x and y}, we choose $x$ and $y$ in $V$ such that $\text{ann}_{\A} \,x = \{ 0 \}$ and $\langle p^{e-1}x,y \rangle \neq 0$. 
It follows that $\text{ann}_{\A}\,y = \{ 0 \}$. 
As above, we set $X = \A x$ and $Y = \A y$.  By hypothesis, $X$ and $Y$ are degenerate, 
so  Lemma~\ref{non-deg criterion} gives 
\[
       \langle \p^{e-1} x, x \rangle =        \langle \p^{e-1} y, y \rangle =  \{ 0 \}.
       \] 
Further, by the argument in \S\ref{x and y},   
$X \cap Y = \{0 \}$. 
We need to show that $X \oplus Y$ is non-degenerate.

Any non-zero element $z \in X \oplus Y$ can be written as $z = p^c g x + p^{c'} g' y$ 
for integers $c$ and $c'$ with $0 \leq c < e,\, 0 \leq c' < e$ and elements $g, g' \in \et$ that are prime to $p$. 
Switching the roles of $x$ and $y$ if necessary, we may assume that $c' \leq c$.  

To prove non-degeneracy of $X \oplus Y$, we will show that $ \langle \p^{e-c-1} x, z \rangle \neq \{ 0 \}$.
Writing $\Bar{f}$ for the image of $f \in \etl$ under the canonical quotient map from $\etl$ to $\A = \etl / \,p^e\,\etl$, we have 
\[
                \langle \p^{e-c-1} x, z  \rangle =   \langle \p^{e-c-1}x, \Bar{p}^c \Bar{g} x + \Bar{p}^{c'} \Bar{g}' y  \rangle. 
                \]
Now   $\Bar{p}^c \, \p^{e-c-1} = \p^{e-1}$ and $\Bar{g} \in \A^\times$, so 
\begin{align*}                
                 \langle \p^{e-c-1}x, \Bar{p}^c \Bar{g} x  \rangle  &= \langle \p^{e-1} x, x \rangle  \\
                                                                                 &= \{ 0 \}.
                                                                                 \end{align*} 
Thus 
\begin{align*}                                          
    \langle \p^{e-c-1} x, z  \rangle      &=    \langle \p^{e-c - 1}x, \Bar{p}^{c'} \Bar{g}'y \rangle  \\                             
                                                         &=    \langle \p^{e-c+c'-1}x, y  \rangle \quad(\text{using}\,\,\Bar{g}' \in \A^\times)  \\     
                                                         &\supset   \langle \p^{e-1} x, y \rangle 
                                                                \quad (\text{as} \,\, c' \leq c, \, \text{so} \,\, e-c+c'-1 \leq e-1) \\
                                                                & \neq \{ 0 \}. 
\end{align*}
In particular, $ \langle \p^{e-c-1} x, z  \rangle \neq \{ 0 \}$, as claimed. 
\end{proof}

\subsection{} 
We have established that $V$ contains an $\A$-submodule of one of the following types:
\begin{enumerate}[(A)]
\item
a non-degenerate $\A$-submodule 
$\A v$ with $\text{ann}_{\A}\,v = \{ 0 \}$; 

\item
a non-degenerate $\A$-submodule as in Lemma~\ref{non-cyclic}.
\end{enumerate}
 Now if $W$ is any non-degenerate 
$\A$-submodule of $V$ then $V = W \oplus W^\perp$ as $\A$-modules. 
Moreover $\text{ann}_{\A} \,W^\perp = \p^c$ for some non-negative integer $c \leq e$.  
If Theorem~A holds for $W$ and $W^\perp$ then it also holds for $V$.
Thus we can complete the proof in Case~II by induction on $\dim_E V$ provided we can establish the result in the two 
special cases  (A) and (B). 

\section{Proof of Theorem~A: Case II-A}  \label{case  II-A}
\subsection{}
This is the cyclic case in which $V = \A v$ with $\text{ann}_{\A}\,v = \{ 0 \}$. That is, the map 
\begin{equation}  \label{one} 
a \mapsto av: \A \to V   
\end{equation} 
is an isomorphism of $\A$-modules. We'll show that there is an anti-unitary involution
$t:V \to V$ such that, for all $a \in \A$,  
\begin{equation} \label{two}
t \, a = \tau(a) \, t 
\end{equation}
as elements of $\text{End}_F \,V$. Now the element $T \in \et$, and so also its image in $\A$, acts on $V$ via $g \in \guv$.  
Thus if we take $a$ to be the image of $T$ in $\A$, then (\ref{two}) gives $t g = \beta g^{-1} t$, or $(tg)^2 = \beta$.  
Hence  
\[
        g = t \cdot t g
        \]
gives the requisite factorization.

\subsection{} \label{anti-unitary}
To establish (\ref{two}), we define $t:V \to V$ by 
\[
     t( a v ) = \tau (a) v, \quad \quad \forall \,\, a \in \A. 
     \]
Thus $t$ is simply the involution $\tau$ of $\A$ transported to $V$ via the isomorphism (\ref{one}). 
It is therefore immediate that $t$ is an involution and that (\ref{two}) holds. To check that $t$ is anti-unitary,  
let $a, b \in \A$. By (\ref{adjoint-tau}), 
\begin{align*}  
     \langle   t(av), t(bv) \rangle   &= \langle   \tau(a) v , \tau(b) v \rangle \\
                                                   &= \langle   b  \tau (a) v, v \rangle \\
                                                   &=  \langle  b v, a v \rangle.
                                                   \end{align*} 

\section{Proof of Theorem~A: Case II-B}
We have $V = \A x \oplus \A y$ with $\langle \p^{e-1} x , y \rangle \neq \{ 0 \}$. Further, $\A x$ and $\A y$ are both degenerate, 
so Lemma~\ref{non-deg criterion} gives 
\[
\langle \p^{e-1} x , x \rangle = \langle \p^{e-1} y, y \rangle =   \{ 0 \}.  
\]
This case requires a more elaborate argument. 

\subsection{}  \label{ss-duality}
We observe first that the subspaces $\A x$ and $\A y$ are in duality via $\vform$.  That is, the map
\begin{equation} \label{duality}
a y \mapsto ( a' x \mapsto \langle   a' x, a y \rangle ):  \A y \to   \text{Hom}_E( \A x, E)
\end{equation}
is a bijection. More precisely, if as in  \S\ref{twist} 
we write $(\A y)^\tau$ for the $E$-vector space structure on $\A y$ obtained by twisting by $\tau$, then (\ref{duality}) 
is an isomorphism of $E$-vector spaces  between  $(\A y)^\tau$ and $ \text{Hom}_E( \A x, E)$. 

To prove this, note that the kernel of the given map is an $\A$-submodule and so equals $\p^c y$ for some non-negative integer 
$c$. Now  $\langle  \p^{e-1} x , y \rangle \neq \{ 0 \}$ and hence $c > e-1$.  As $\p^e = \{ 0 \}$, the kernel must be  trivial and  
thus (\ref{duality}) is injective. Since $\dim_E \A x = \dim_E \A y \,\, (= \dim_E \A)$, the map is also surjective. 

\subsection{}
The map $a  x  \mapsto \langle y, \tau(a) x \rangle = \langle a y, x \rangle $ belongs to $\text{Hom}_E( \A x, E)$. 
Thus by \S\ref{ss-duality}, there is a unique $\gamma \in \A$ such that 
\begin{equation}   \label{key}
             \langle  a y ,  x \rangle =  \langle a x, \gamma y \rangle, \quad \quad \forall \,\, a \in \A.
             \end{equation} 
We claim that $\gamma \in \A^\times$. Indeed,  $\langle \p^{e-1} x , y \rangle \neq \{ 0 \}$ and $\tau(\p) = \p$, so 
$\langle \p^{e-1} y, x \rangle \neq \{ 0 \}$. It follows that $\langle \p^{e-1} x, \gamma y \rangle \neq \{ 0 \}$, or equivalently
$\langle \tau ( \gamma) \p^{e-1} x, y \rangle \neq \{ 0 \}$.  As $\p^e = \{ 0 \}$, we see that $\tau (\gamma) \notin \p$. Therefore
  $\tau(\gamma) \in \A^\times$, whence also $\gamma \in \A^\times$.

\subsection{}
We claim next that $ \gamma \,  \tau (\gamma) = 1$.  Rewriting (\ref{key}) as 
\[
\langle  x, a \gamma y \rangle =   \langle y, a x \rangle, 
\]
we have
\begin{align*} 
  \langle  x, a \gamma y \rangle      &= \ep\, \tau (  \langle a x , y \rangle ) \\
                                                        &= \ep\,  \tau ( \langle a x, \gamma^{-1} \gamma y \rangle ) \\
                                                        &= \ep \, \tau  (  \langle \tau(\gamma^{-1} ) a x, \gamma y  ) \\
                                                        &=  \ep \, \tau ( \langle \tau(\gamma^{-1} a y, x \rangle ) \quad (\text{by} \,\,  (\ref{key})) \\
                                                        &=    \langle x, \tau (\gamma^{-1}) a y \rangle, \qquad \forall \,\, a \in \A.  
                                                        \end{align*} 
It follows that 
\[
\langle  a x, \gamma y   \rangle =    \langle a x, \tau(\gamma^{-1}) y \rangle, \quad \quad \forall \,\, a \in \A.
\]
By bijectivity of (\ref{duality}), $\tau(\gamma^{-1}) y = \gamma y$, whence $\tau (\gamma^{-1} ) = \gamma$, that is, 
$\gamma \, \tau(\gamma) = 1$.

\subsection{}
Define $t: \A x \oplus \A y \to \A x \oplus \A y$ by  
\[
       t( ax + by) =  \tau(a) u + \tau (b) \gamma y, \qquad \forall \,\,a, b \in \A.
       \]
We claim that $t$ is an anti-unitary involution such that, for all $a \in \A$,   
\begin{equation} \label{three}
t \, a = \tau(a) \,  t  
\end{equation} 
as elements of $\text{End}_F (\A x \oplus \A y)$.  
Once this is established, we can complete the argument exactly as in \S\ref{case II-A}. 
That is, $(tg)^2 = \beta$, and thus as above 
$g = t \cdot tg$ gives the requisite factorization. 

Applying $t$ twice, we obtain 
\begin{align*} 
      ax + by        &\overset{t}{\longmapsto}   \tau(a) x + \tau (b) \gamma y \\
                          &\overset{t}{\longmapsto}   a x +  b \tau( \gamma) \,\gamma y 
                           =     a x + by   \quad (\text{as}\,\, \tau(\gamma) \, \gamma = 1), 
                           \end{align*} 
and so $t$ is an involution.                           

 The identity (\ref{three}) is immediate. In detail, for all $a, a', b' \in \A$, 
 \begin{align*} 
    t a ( a'x + b'y)    &=   t( a a' x + a b' y) \\
                              &= \tau(a) \tau(a') x + \tau(a) \tau (b') \gamma y \\
                              &=  \tau (a) t ( a'x + b'y). 
                              \end{align*} 
Finally, to show that $t$ is anti-unitary, it suffices to verify the following four identities (for all $a, b \in \A$):
\begin{align*}
(1) \,\, \langle t( ax) , t(bx) \rangle &=   \langle   bx, ax \rangle; \\
(2) \,\, \langle t( ay) , t(by) \rangle &= \langle   by, ay \rangle; \\
(3)  \,\, \langle t(ax), t(by) \rangle &= \langle by, ax \rangle;  \\
(4) \,\, \langle t(by),  t( ax) \rangle &= \langle ax, by \rangle.
\end{align*}      
Applying $\tau$ to both sides of $(3)$ gives  $(4)$, so it's enough  to check (1), (2), (3).      
We can verify (1) directly as in \S\ref{anti-unitary}.  The argument for (2) is similarly 
straightforward using $\gamma \, \tau(\gamma) = 1$. 
To check (3), note
\begin{align*} 
    \langle  t( ax), t(by) \rangle    &=  \langle \tau(a) x, \tau (b) \gamma y \rangle \\
                                                  &=    \langle b \tau(a) x, \gamma y \rangle \\
                                                  &=    \langle b \tau(a) y, x \rangle \quad (\text{using} \,\, (\ref{key}) )\\
                                                  &=   \langle by, ax \rangle.    
                                                                        \end{align*}
This completes the proof in Case~II-B and so concludes the proof of Theorem~A. 
\qed 

\section{Some Examples and Applications} \label{apps}

\subsection{} Suppose that $E=F$ and $\ep = -1$, so that $\uv = \mathrm{Sp}(V)$ and $\mathrm{GU}(V) = \mathrm{GSp}(V)$. 
Assume also that  $\mathrm{char}(F) \neq 2$. As noted in the introduction, Theorem~A for the symplectic  group $\mathrm{Sp}(V)$ was proved by 
Wonenburger \cite{Wo} and the case of the similitude group $\mathrm{GSp}(V)$ was treated in  \cite{ryan-gsp}. 
Assume now that $\mathrm{char}(F) = 2$. Then the case of symplectic groups was proved by Gow \cite{Gow} and Ellers and Nolte \cite{EN}. 
If $F$ is perfect, the similitude case follows readily (as every element of $F$ is a square). The similitude case for $\mathrm{char}(F) = 2$ and 
$F$ imperfect appears to be new.

\subsection{} Suppose now that $[E:F]=2$ and $\epsilon = 1$. 
Let $V = E^n$ and view the elements of $V$ as column vectors. 
Consider the non-degenerate hermitian form  
$\vform$ on $V$ given by $\langle x, y \rangle = {^\top x} \, \tau(y)$. Here, as in \S2.1, 
$\tau(y)$ is obtained by applying the automorphism $\tau$ to each coordinate of $y$. Similarly, 
for any matrix $a = [a_{ij}]$ with entries in $E$, we set $\tau(a) = [\tau(a_{ij})]$. 
We write $\mathrm{U}(n)$ for the isometry group of $\vform$. Thus  
\[ 
\mathrm{U}(n) = \{ g \in \mathrm{GL}_n(E) :  {^\top g} \, \tau(g) = 1 \}.
\] 
The map $x \overset{c}{\longmapsto}  \tau(x):V \to V$ is an anti-unitary involution for $\vform$. 
For any $a \in M_n(E)$ (viewed an an $E$-linear map on $V$ via left multiplication), we have ${}^{c} a = c ac^{-1} = \tau(a)$. 
 
The calculation that gave ($\ast$) of \S2.1 shows that Theorem~A for $\mathrm{U}(n)$ is equivalent to the statement: 
\begin{enumerate}[$(\ast\ast)$]
\item
if $g \in \mathrm{U}(n)$ then $g =   s_1 s_2$ for elements $s_i \in \uv$ such that $s_i \,{}^{c}s_i = 1$ for $i=1, 2$.  
\end{enumerate}
\smallskip
From $s_i \,{}^{c}s_i = 1$ and $s_i \in \mathrm{U}(n)$, we see that 
$$s_i^{-1} = {^c s_i} = \tau(s_i) = {^\top s_i^{-1}}.$$
Thus each $s_i \in \mathrm{U}(n)$ is symmetric as an element of $\mathrm{GL}_n(E)$. Hence  ${^\top g} = {^\top s_2} {^\top s_1} = s_2 s_1$ 
and 
\[
s_1^{-1}  g \, s_{\mathstrut 1} = {^\top g}. 
\]
In particular, we obtain the following unitary group version of the result of Frobenius (used in \S\ref{frob})  that any matrix is conjugate to its transpose by a symmetric matrix. 

\begin{Cor}
For any $g \in \mathrm{U}(n)$, there exists a symmetric matrix $s \in \mathrm{U}(n)$ such that $s g s^{-1} = {^\top g}$.
\end{Cor}

When $E/F = \mathbb{C}/\mathbb{R}$, the Corollary follows immediately from the fact that any unitary matrix is unitarily diagonalizable. 
In the case that $E/F$ is an extension of finite fields,  
the Corollary is proved in Lemma 5.2 of \cite{GowVin}.  

\subsection{} \label{GOdet} Let $E=F = \mathbb{F}_q$ be a finite field with $q$ elements with $q$ odd and let $\epsilon = 1$, so that
$\mathrm{GU}(V)$ is a finite group of orthogonal similitudes.  We restrict attention to the case that $\mathrm{dim}(V)=2m$ is even. In this 
setting  there are two equivalence classes of non-degenerate symmetric forms on $V$, giving two distinct finite orthogonal similitude groups.  We denote these groups  by $\mathrm{GO}^{\pm}(2m, \mathbb{F}_q)$, and write $\mathrm{O}^{\pm}(2m, \mathbb{F}_q)$ for the 
corresponding orthogonal groups.  For $\uv = \mathrm{O}^{\pm} (2m, \mathbb{F}_q)$, the element $h_1$ in Theorem~A can be chosen so that 
$\mathrm{det}(h_1) = (-1)^m$ (see Lemma 4.7 of \cite{SFV}).  For use in later work, we now extend this observation to the case 
$\guv = \mathrm{GO}^{\pm}(2m, \mathbb{F}_q)$.

\begin{Prop}  Let $G = \mathrm{GO}^{\pm}(2m, \mathbb{F}_q)$ with $q$ odd and let $g \in G$ with $\mu(g) = \beta$.  Then there exist $h_1, h_2 \in G$ such that $g = h_1 h_2$, $\mu(h_1) = 1$, $\mu(h_2) = \beta$, $h_1^2 = 1$, $h_2^2 = \beta$, and $\det(h_1) = (-1)^m$.
\end{Prop}
\begin{proof} The case $\mu(g) = 1$ is implied by Lemma 4.7 of \cite{SFV}.  If $\mu(g) = \beta$ is a square in $\mathbb{F}_q$, say $\beta = \gamma^2$, then $g' = \gamma^{-1} g$ satisfies $\mu(g') = 1$, so we may write $g' = h_1 h'$ with $h_1$ and $h'$ orthogonal involutions, and $\det(h_1) = (-1)^m$.  Now let $h_2 = \gamma h'$, so that $g = h_1 h_2$ satisfies the desired conditions.  We now assume that $\mu(g) = \beta$ with $\beta$ a non-square in $\mathbb{F}_q$.

We proceed by considering Cases I, II-A, and II-B, as in the main result proved above.  In Case I, we have $V = V_1 \oplus V_2$, where $\mathrm{dim}(V_1) = \mathrm{dim}(V_2) = m$.  In this scenario, we have $E=F$ and $\epsilon = 1$, and in 2.2, the element $s_1$ satisfies $\det(s_1) = (-1)^m$ since $d_1$ is symmetric.  Taking $s_1 = h_1$ and $s_2 = h_2$ gives the desired factorization in this case.

To handle Case II, we appeal to the description of conjugacy classes in $\mathrm{GO}^{\pm}(2m, \mathbb{F}_q)$, as described by Shinoda in Section 1 of \cite{Shinoda}.  In particular, it is proven there that Case II-B occurs if and only if the minimal polynomial $p(T)^e$ of $g$ on $V$ is of the form $(T^2 - \beta)^e$ where $e = 2k-1$ is an odd positive integer.  This statement is contained in (1.18.2) of \cite{Shinoda}.  Note that Wonenburger makes mention of the parallel exceptional cases which occur for the $\beta = 1$ case in Remark I of \cite{Wo}.

We now apply some calculations made in \cite{ryan-gsp, ryan-go}.  Consider first Case II-A, where we have $V = \mathcal{A} v$ is cyclic, and the minimal polynomial for $g$ on $V$ is of the form $p(T)^e$, and which is not of the form $(T^2 - \beta)^{2k-1}$.  In particular, it follows from the fact that $\tau(p) = up$ for some $u \in F[T, T^{-1}]^{\times}$ that $p(T)$ has even degree, and let $2m = e \, \mathrm{deg}(p) = \mathrm{dim}(V)$.  Now define
$$ P = \mathrm{span} \{ (g^i + \beta^i g^{-i})v \, \mid \, 0 \leq i < m \}, \quad \text{ and } \quad Q = \mathrm{span} \{ (g^i - \beta^i g^{-i})v \, \mid \, 0 < i \leq m \}.$$
In Proposition 3(i) of \cite{ryan-gsp} and in Theorem 1 of \cite{ryan-go}, it is shown that $V = P \oplus Q$, and if we define $h_1$ to have $+1$-eigenspace $P$ and $-1$-eigenspace $Q$ and $h_2 = h_1 g$, then we have $h_1, h_2 \in G$ with $\mu(h_1) = 1$, $\mu(h_2) = \beta$, $h_1^2 = 1$, and $h_2^2 = \beta$.  Since $\mathrm{dim}(Q) = (-1)^m = \det(h_1)$, this gives the desired factorization.

Finally, consider Case II-B, where we have $V = \mathcal{A}x \oplus \mathcal{A}y$, and as mentioned above, the minimal polynomial for $g$ must be of the form $(T^2 - \beta)^{2k-1}$.  In this case, we have $\mathrm{dim}(V) = 2m$, where $m = 4k-2$.  Define
$$ P_x = \mathrm{span} \{ (g^i + \beta^i g^{-i})x \, \mid \, 0 \leq i \leq 2k-1 \} \quad \text{ and } \quad Q_x = \mathrm{span} \{ (g^i - \beta^i g^{-i})x \, \mid \, 0 < i < 2k-1 \},$$
and define $P_y$ and $Q_y$ analogously.  In Proposition 3(i) and (iii) of \cite{ryan-gsp} and in Theorem 1 of \cite{ryan-go}, it is shown that if $P = P_x \oplus Q_y$ and $Q = Q_x \oplus P_y$, and we define $h_1$ to have $+1$-eigenspace $P$ and $-1$-eigenspace $Q$ and $h_2 = h_1 g$, then we again have $h_1, h_2 \in G$ with $\mu(h_1) = 1$, $\mu(h_2) = \beta$, $h_1^2 = 1$, and $h_2^2 = \beta$.  Since $\mathrm{dim}(P_y) = 2k$ and $\mathrm{dim}(Q_x) = 2k-2$, then $\mathrm{dim}(Q) = 4k-2 = m$, so $\det(h_1) = (-1)^m$.  
\end{proof}

\section{Proof of Theorem~B}  \label{main result}

Recall that $h \in {\rm Aut}_F(V)$ is an anti-unitary involution and that $\ui g = \mu(g)^{-1} hgh^{-1}$ for $g \in \guv$. 
Thus $\iota$ is  a continuous automorphism of $\guv$ of order two. 
The restriction  $\iota |_{ \uv}$ gives the automorphism  $g \mapsto hgh^{-1}$ of $\uv$ 
which we again denote by $\iota$.  We restate our main result.

\begin{ThmB}  
The maps $\iota:\uv \to \uv$ and $\iota:\guv \to \guv$ are dualizing involutions. 
\end{ThmB}

We recall some character theory in \S\ref{chars}.  Using this, we will see in  \S\ref{conjugacy} that Theorem~B  follows almost immediately from
Theorem~A. 

\subsection{}    \label{chars}
Let $G$ be the $F$-points of a reductive algebraic $F$-group.  As usual, we write 
$C_c^\infty(G)$ for the space of complex-valued functions on $G$ that are locally constant and of compact support. 
Let $(\pi, V)$ be a smooth representation of $G$. For $f \in C_c^\infty(G)$, the operator $\pi(f):V \to V$ is given by 
\[
\pi(f)v  = \int_G f(g) \pi(g) v \, dg, \quad \quad v \in V, 
\]
where the integral is with respect to a Haar measure on $G$ which we fix once and for all. 
Assume now that $(\pi,V)$ is irreducible. It is well-known that $(\pi,V)$ is then
{\it admissible} \cite{Jac}, that is, the space $V^K$ of $K$-fixed vectors has finite dimension for any open subgroup $K$ of $G$.  It follows 
that the image of $\pi(f)$ has finite dimension and thus $\pi(f)$ has a well-defined trace. The resulting linear functional $f \mapsto \tr \, \pi(f): C_c^\infty(G) \to \cx$ is called the {\it distribution character} of $\pi$. It determines the irreducible representation $\pi$ up to equivalence (\cite{BZ} 2.20).

It is straightforward to check that $ \tr \, \pic(f) =  \tr \, \pi(\fc)$ where $\fc(g) = f(g^{-1})$ for $g \in G$.

Let $G_{\rm reg}$ denote the set of regular semisimple elements in $G$. By \cite{H-C, Ad-Kor}, 
the distribution character of $\pi$ is represented by a locally constant function  $\Theta_\pi$ 
on $G_{\rm reg}$ called the {\it character} of $\pi$. That is, 
\begin{equation} \label{H-C-char}
     \tr \, \pi(f)   = \int_G   f(g) \,\Theta_\pi (g)  \, dg, \quad \quad  f \in C_c^\infty(G). 
     \end{equation}

\begin{Rmk}     
Existence of $\Theta_\pi$ is established in \cite{H-C} for arbitrary connected reductive $F$-groups based on the submersion principle of its title.  
Harish-Chandra, however, only gave a proof of the principle in characteristic zero with a comment that a general proof was known. 
A full proof (due to G.~Prasad) appears in Appendix~B to \cite{Ad-DeB}. In \cite{Ad-Kor} \S13, Adler and Korman explain how to extend 
Harish-Chandra's and Prasad's arguments to non-connected reductive $F$-groups. 
\end{Rmk}

By (\ref{H-C-char}), the function $\Theta_\pi$ determines the distribution character of $\pi$ and thus $\pi$ is determined up to equivalence
by $\Theta_\pi$. In the same way,  $\Theta_\pi$ is constant on (regular semisimple) conjugacy classes. 
From $\tr \, \pic (f) = \tr \, \pi (f^\vee)$ for $f \in C_c^\infty(G)$, we also have 
$
\Theta_{\pi^{\vee}}(g) = \Theta_\pi(g^{-1})
$      
for $g \in G_{\rm reg}$, again by (\ref{H-C-char}).

\subsection{}    \label{conjugacy}
For  $\pi$ a smooth representation of $G$ and $\alpha$ a continuous automorphism of $G$, we write $\pia$ 
for the smooth representation given by $\pia(g) = \pi (\ua g)$ for $g \in G$.   

For any $g \in \guv$, we noted in the introduction that the elements $\ui g$ and $g^{-1}$ are conjugate by an element of $\uv$.  
To prove Theorem~B, it suffices therefore to observe the following. 
 \begin{Lem}
Let $\alpha$ be a continuous automorphism of $G$ such that $\ua g$ is conjugate to $g^{-1}$ for any $g \in G$. 
Then $\pia \cong \pic$ for any  irreducible smooth representation $\pi$ of $G$. 
\end{Lem}
 
 \begin{proof}
The main detail to check is that a continuous automorphism $\gamma$ of $G$ preserves the Haar measure $\mu_G$ on $G$.
We have $\mu_G \circ \gamma = c_\gamma \,\mu_G$ for some  $c_\gamma > 0$. 
Writing ${\rm Aut}_{\rm c}(G)$ for the group of continuous automorphisms of $G$ and $\mathbb{R}^\times_{\rm pos}$ 
for the multiplicative group of positive real numbers, the assignment
$\gamma \mapsto c_\gamma: {\rm Aut}_{\rm c}(G) \to \mathbb{R}^\times_{\rm pos}$ is a homomorphism.
Let $K$ be a compact subgroup of $G$ of maximal volume. (Note $K$ exists as $G$ has a finite non-zero number of 
conjugacy classes of maximal compact subgroups.) 
For any $\gamma \in {\rm Aut}_{\rm c}(G)$, we have $\mu_G(\gamma(K)) = c_\gamma \,\mu_G(K)$, so that 
$c_\gamma \leq 1$. Similarly $c_{\gamma^{-1}} = c_\gamma^{-1}  \leq 1$. Hence $c_\gamma = 1$, as required.

In particular,  $\alpha$ preserves the Haar measure on $G$. Thus, for any  irreducible smooth representation $\pi$ of $G$,  
\begin{align*} 
\pia (f)   &= \int_G f(g) \, \pi (\ua g) \,dg  \\
             &=  \int_G f({}^{\alpha^{-1}} g ) \, \pi (g) \, dg, \quad \quad  f \in C_c^\infty(G).  
             \end{align*} 
That is,  $\pia(f) = \pi(\ua f)$ for $f \in C_c^\infty(G)$ where $\ua f(g) = f({}^{\alpha^{-1}}g)$. It follows that  
$\tr \,\pia(f) = \tr \, \pi(\ua f)$, so that 
\begin{align*}
  \int_G f(g) \, \Theta_{\pia} (g) \,dg  &= \int_G  f({}^{\alpha^{-1}} g ) \, \Theta_\pi(g) \, dg \\  
                                                     &= \int_G f(g) \, \Theta_\pi (\ua g) \,dg,   \quad \forall \,\, f \in C_c^\infty(G).
                                                     \end{align*}
Therefore $ \Theta_{\pia} (g) = \Theta_\pi (\ua g)$ for $g \in G_{\rm reg}$. As characters are constant on conjugacy classes, it follows that 
$\Theta_{\pia}(g)  = \Theta_\pi(g^{-1})$ for $g \in G_{\rm reg}$. Thus $\Theta_{\pia} = \Theta_{\pic}$ and $\pia \cong \pic$. 
 \end{proof}

\subsection{}
We record a direct consequence of Theorem~B, well-known to experts (see, for example, \cite{DP} page 305). 
Suppose $E=F$ so that $\vform$ is orthogonal or symplectic.  We change notation slightly and write 
${\rm O}(V)$ and ${\rm GO}(V)$ or ${\rm Sp}_{2n}(F)$ and ${\rm GSp}_{2n}(F)$  
(where $\dim_F V = 2n$) for the resulting isometry and similitude groups. 
The center of each similitude group consists of scalar transformations. Dividing by this center gives the corresponding 
projective groups   ${\rm PGO} (V)$ and ${\rm PGSp}_{2n}(F)$.

\begin{Cor}
\begin{enumerate}[a.] 
\item
Every irreducible smooth representation of ${\rm O(V)}$  is self-dual.

\item
If $-1 \in (F^\times)^2$ then every  irreducible smooth representation of ${\rm Sp}_{2n}(F)$ is self-dual.

\item
For any irreducible smooth representation $\pi$ of ${\rm GO(V)}$ or ${\rm GSp}_{2n}(F)$, 
$\pic \cong \pi \otimes  \omega_\pi  \circ \mu^{-1} $ where $\omega_\pi$ denotes the central character of $\pi$.
In particular, every irreducible smooth representation of ${\rm PGO(V)}$ or ${\rm PGSp}_{2n}(F)$ is self-dual.

\end{enumerate}

\end{Cor}

\begin{proof}
Part~a is immediate as $h \in {\rm O}(V)$,  so $\iota:{\rm O}(V) \to {\rm O}(V)$ is inner. 

For part~b, it suffices to note that 
$\iota(g) = h g h^{-1}$ defines an inner automorphism  of ${\rm Sp}_{2n}(F)$ for any 
anti-unitary (i.e., skew-symplectic) $h \in {\rm GSp}_{2n}(F)$.  Given $i \in F^\times$ with $i^2 = -1$,  
the homothety $i$ satisfies $\mu(i) = i^2 = -1$ and thus  $i h \in {\rm Sp}_{2n}(F)$.   
Since $\ui g = (i h) g (i h)^{-1}$ for $g \in {\rm Sp}_{2n}(F)$, we see that $\iota$ is inner.  

For part~c, observe that  $g \mapsto \mu(g)^{-1} g$ defines a dualizing involution of  each similitude group.
\end{proof}

\section{Dualizing involutions do not always exist}  \label{non-existence}

Let $D$ be a central $F$-division algebra of dimension $m^2$ over $F$. Let $n$ be a positive integer and set $G = {\rm GL}_n(D)$.
 We show that $G$ can admit an automorphism that takes each irreducible smooth representation to its dual only in the known cases
$m=1$ \cite{GK, Tupan} and $m=2$ \cite{MS, Ragh}.  Thus it is only in these two cases that
$G$ can admit an automorphism $\theta$ such that $\ut g$ is conjugate to $g^{-1}$ for all $g \in G$, an observation also made by 
Lin, Sun and Tan  (\cite{LinSunTan} Remark (c) page 83). In fact, the two statements 
-- non-existence of automorphisms that take each irreducible smooth representation to its dual and non-existence of automorphisms that invert each conjugacy class -- must surely be equivalent. 

\begin{Prop}
Suppose there exists an automorphism $\theta$ of $G$ such that $\pit \simeq \pic$ for all irreducible smooth representations $\pi$ of $G$. 
Then $D = F$ or $D$ is a quaternion algebra over $F$ (equivalently,  $m = 1$ or $2$).  
\end{Prop}

\subsection{} \label{conts}
We need a preliminary observation. 

Let $\fo$ denote  the valuation ring in $F$ and $\fp$ the unique maximal ideal in $\fo$.

\begin{Lem}   \label{prelim}
Any  field automorphism of $F$ preserves $\fp$. In particular, field automorphisms of $F$ are automatically continuous.
\end{Lem}

\begin{proof}

Write $q$ for the cardinality of the residue field
$\fo / \fp$ and $v_F$ for the normalized valuation on $F$.
The ideals $\fp^k$ (for $k$ a positive integer) form a neighborhood basis of $0 \in F$. 
Thus an automorphism that preserves $\fp$ is continuous. 

Writing $p$ for the residual characteristic of $F$, the set $1+\fp$  can be characterized algebraically as follows:   
\begin{itemize}
\item[]
$x \in 1+\fp$ if and only if $x$ admits an $n$-th root (i.e., there is a $y \in F^\times$ with $y^n = x$) for any $n$ such that $p \nmid n$ . 
\end{itemize}
Indeed, using Hensel's Lemma or simply that  $1+\fp$ is a pro-$p$-group, one sees  
that each element of $1+\fp$ admits an $n$-th root for any $n$ such that $p \nmid n$. 
In the other direction, suppose $x$ has this property. 
Then $n$ divides $v_F(x)$  for infinitely many integers $n$, whence $v_F(x) = 0$, i.e., $x \in \fo^\times$. 
Let $y$ be a $(q-1)$-th root of $x$.   Then $y \in \fo^\times$ and the relation  $y^{q-1} =x$ implies $x \in 1+\fp$. 

It follows that any field automorphism of $F$ preserves $1+\fp$ and so also $\fp$. 
\end{proof}

\subsection{}
{\it Proof of Proposition.}
We use the isomorphism  $x \mapsto x 1_n:F^\times \to Z(G)$
to view the central character $\ompi$ of any smooth irreducible representation $\pi$ of $G$ as a smooth character of $F^\times$.

Suppose first that $D$ is not isomorphic to its opposite $\dop$. We appeal to  Dieudonn\'{e}'s description of the automorphism groups 
of general linear groups over division algebras \cite{D}. In the case at hand, this gives a) a homomorphism $\eta:G \to F^\times$, 
b) an automorphism $\sigma$ of $D$ acting on $G$ via ${}^{\sigma}(a_{ij}) = ({}^{\sigma}a_{ij})$ and c) an element $h \in G$ such that 
\begin{equation} \label{aut}
 \ut g = \eta(g) \, h   \, {}^{\sigma}g \, h^{-1},    \quad \,\,\,  g \in G. 
 \end{equation} 
(See \cite{D} Theorems 1 and 3 for the case $n \geq 3$ and the end of \emph{ibid.} \S12 for the case $n=2$.) 

As $\pit \simeq \pic$, we have $\ompi \circ \theta= \ompi^{-1}$ (for all smooth irreducible representations $\pi$). It follows that 
\[
     \ut a = a^{-1}, \quad \,\,\,  a \in F^\times.
\]
Thus, by (\ref{aut}),
\[
    a^{-1}  = \eta(a) \, {}^{\sigma}a, \quad \,\,\,   a \in F^\times.
     \]
 
We have $G/(G,G) \simeq \dx / (\dx,\dx)$ via  Dieudonn\'{e}'s non-commutative determinant $\text{Det}$. 
Further, the reduced norm  $\nrd$ from $D$  to $F$ induces an isomorphism $\dx / (\dx,\dx) \simeq \fx$. Thus there is 
a character $\eta_1:\fx \to \fx$ such that $\eta(g) = \eta_1(\nrd \circ \text{Det} \,g )$, for $g \in G$. 
Using $\text{Det} \, a = a^n (\dx,\dx)$ and $\nrd \,a = a^m$, for $a \in \fx$, it follows that      
\[
            a^{-1}  = \eta_1(a)^{mn} \, {}^{\sigma}a, \quad \,\,\,  a \in F^\times.
\]
Taking $a= \varpi$,  a uniformizer in $F$,  and applying  $v_F$, we obtain
\[
    -1 = mn v_F(\eta_1(\varpi)) + v_F({}^{\sigma}\varpi)
    \]
By Lemma~\ref{conts}, $ v_F({}^{\sigma}\varpi) = 1$,  and hence $m \mid 2$.  Thus $D = F$ or $D$ is a quaternion algebra over $F$ which 
contradicts our assumption that $D$ is not isomorphic to $\dop$. 

It follows that there is an isomorphism $\alpha:D \to \dop$. If $\alpha$ is $F$-linear, then $D$
represents an element of order at most two in the Brauer group of $F$. As the only such elements are the class of $F$ and the class of the unique quaternion division algebra over $F$, the result follows. In general, however, we can only say that $\alpha$ preserves the center $F$ of $D$.  
By Lemma~\ref{prelim}, it must also preserve $\fo$. The ring $D$ contains a unique maximal $\fo$-order $\od$ consisting of 
the elements of $D$ that are integral over $\fo$. From this description, we see that $\alpha$ preserves $\od$. Thus $\alpha$ also 
preserves the unique maximal (left or right) ideal $\pd$ in $\od$, and hence induces an automorphism of the quotient $\od / \pd$, 
a finite field of order $q^m$.  Let $\varpi_D$ be a generator of $\pd$, i.e.,  
$\pd = \varpi_D \, \od = \od \, \varpi_D$.  Then, for $D \neq F$, there is a unique integer $r$ with $1 < r < m$ and $(r,m) = 1$ such that 
\begin{equation} \label{Hasse} 
           \varpi_D \, x \, \varpi_D{}^{-1}  \equiv x^{q^r} \pmod{\pd}, \quad \,\,\,  x \in \od.
\end{equation}
Moreover the congruence is  independent of the choice of generator $\varpi_D$.  (This all follows,  for example, from \cite{Re} 14.5.) 
Applying $\alpha$ to (\ref{Hasse}) and rearranging  (and using the fact that $\od / \pd$ has order $q^m$), we obtain
\[
  \alpha(\varpi_D) \, x \, \alpha(\varpi_D)^{-1}  \equiv x^{q^{m-r}}  \pmod{\pd}, \quad \,\,\,  x \in \od.
\] 
Since (\ref{Hasse}) holds for all generators of $\pd$, we deduce that $r = m-r$ or $2r = m$, whence $r=1$ and $m=2$.  Thus
$D$ is a quaternion algebra over $F$ and we have completed the proof. 
\qed

\end{document}